\documentclass{article}
\usepackage{amsmath}

\usepackage{amssymb}
\usepackage{amsthm}

\newtheorem{lemma}{Lemma}[section]
\newtheorem{theorem}{Theorem}[section]

\begin{document}

\title{Limit Weierstrass Points on the Kenyon-Smillie Family of Plane Quartics}


\author{R. F. Lax}

\maketitle




\begin{abstract}
Costantini and Kappes gave an algebraic equation of the universal family over the Kenyon-Smillie (2,3,4)-Teichm\"uller curve.  This equation gives rise to a family of projective plane quartic curves with three singular members.  These singular curves are: (1) an integral nodal quartic, (2) the union of a line and a nodal cubic, and (3) a quadruple line.  We determine the limit Weierstrass points on these singular curves.
\end{abstract}

\maketitle


A Teichm\"uller curve is a rare subvariety of the moduli space of curves that has several different characterizations (cf. \cite{Moller}), one of which is: an immersed curve in the moduli space $\mathcal{M}_g$ that is totally geodesic for the Teichm\"uller metric.  It is also the image in $\mathcal{M}_g$ of a closed GL$_2^+(\mathbb{R})$-orbit in the moduli space $\Omega\mathcal{M}_g$ parametrizing flat surfaces.  One of the few sporadic primitive Teichm\"uller curves is known as the Kenyon-Smillie (2,3,4)-curve, constructed by unfolding and shearing a Euclidean triangle with angles $2\pi/9,3\pi/9,\hbox{ and }4\pi/9$.
M. Costantini and A. Kappes \cite{coskap} gave an algebraic equation of the universal family over the Kenyon-Smillie (2,3,4)-Teichm\"uller curve.  Put

\begin{align}\label{family}
G(X,Y,Z)=&X^4 -3X^3Y +6X^3Z -3X^2Y^2 -6X^2YZ +6X^2Z^2 +\notag \\&4XY^3-6XY^2Z -6XYZ^2 +XZ^3 +3Y^4 +3Y^3Z\notag \\
F(X,Y,Z,t)=& X^4 +tG(X,Y,Z)
\end{align}

\begin{theorem}[\cite{coskap}]
The universal family over the complement of the orbifold point of the Kenyon-Smillie (2,3,4)-Teichm\"uller curve is given by the family of plane quartics satisfying the equation $F(X,Y,Z,t)=0$,
where $t$ varies in ${\mathbb{P}}^1\setminus \{0,1,\infty\}$.
\end{theorem}
  In this note, we determine the limit Weierstrass points for the singular curves in the family $F(X,Y,Z,t)=0, t\in {\mathbb{P}}^1$.  Since we are dealing with plane quartics, limits of Weierstrass points are the same as limits of flexes.

  The three singular fibers are quite different.  When $t=1$, we have an integral rational nodal plane quartic.  Such curves were studied by the author and C. Widland \cite{LaxWid} and the Weierstrass points may be defined intrinsically on such a curve using the (canonical) bundle of dualizing differentials.
  
   When $t=\infty$, we have a reducible, but reduced, curve with a rational component and a nodal cubic component.  The notion of limit Weierstrass points on reducible curves was studied by D. Eisenbud and J. Harris \cite{EHLimitsWpts} for curves of compact type and has been investigated by other authors for more complicated reducible curves.  However, the case in which the reducible curve has a rational component is a difficult one, with some results having been obtained by E. Esteves \cite{EstCD}, and by E. Esteves and N. Medeiros \cite{EstMed}.
   
   When $t=0$, we have the nonreduced curve $X^4=0$.  While some authors have investigated some curves with multiple components (for example in \cite{GeomSch} and \cite{EstMed}), we are not aware of any examples in the literature of limit Weierstrass points on a multiple line.   

\section{The fiber over $t=1$}

From \cite{coskap}, we have that $F(X,Y,Z,1)=0$ is an integral curve with three nodes.  Weierstrass points on integral Gorenstein curves were defined in \cite{WidLax} using the (canonical) line bundle of dualizing differentials.  The limit Weierstrass points on such a fiber in a family are these intrinsically defined Weierstrass points by \cite{LaxPAMS}.  Put
$$f_1(x,y) = F(X,Y,1,1).$$

Using Macaulay2 \cite{M2}, a Groebner basis, with respect to graded reverse lex, for the ideal $I=(f_1(x,y),\frac{\partial f_1}{\partial x}(x,y),\frac{\partial f_1}{\partial y}(x,y))$ is 
$$\{ y^2-x+2y-1, xy+x+y-1, x^2+x+y-1 \}.$$
From this we see that at a singular point of $f_1(x,y)=0$, we must have $x^3-3x+1=0$.  The roots of this equation are $\nu, \nu^2-2$, and $2-\nu-\nu^2$, where $\nu= 2\cos (2\pi/9)$.  Thus, the nodes of $f_1(x,y)=0$ are $$ (\nu,1-\nu-\nu^2), (\nu^2-2,\nu-1), \hbox{ and } (2-\nu-\nu^2,\nu^2-3).$$
We know from \cite{LaxWid} that each of these nodes is a Weierstrass point of weight 6, 7 (if a flecnode), or 8 (if a biflecnode).

\begin{lemma}
Let $C_t$ denote the projective plane curve $F(X,Y,Z,t)=0$.  If $t\ne 0,\infty$, then the point $(0:-1:1)$ is a hyperflex (Weierstrass point of weight 2) on $C_t$ and the point $(0:0:1)$ is a flex on $C_t$.
\end{lemma}

\begin{proof}
Note that
\begin{equation}\label{Gfactor}
G(X,Y,Z) = (X+Y+Z)(X^3-4X^2Y+5X^2Z +XY^2 -7XYZ +XZ^2 +3Y^3).
\end{equation} 
The point $(0:-1:1)$ is on $C_t$ and, assuming $t\ne 0,\infty$, has tangent line $X+Y+Z=0$. This tangent line meets $C_t$ only at this point, hence has intersection multiplicity four at this point.  Therefore, $(0:-1:1)$ is a hyperflex and a Weierstrass point of weight 2 on $C_t$ for $t\ne 0,\infty$.  It is straightforward to see that $(0:0:1)$ is a common point of $C_t$ and its Hessian.
\end{proof}

Since the curve $C_1$ is integral we can use its Hessian to find its flexes (with multiplicities).  We will see that all the flexes of $C_1$ occur in the affine $xy$-plane.

  Let $\Phi(X,Y,Z)$ denote a homogeneous polynomial of degree $d>1$ and let $\phi(x,y)= \Phi(x,y,1)$ be its dehomogenization. Then the dehomogenization of the Hessian of $\Phi$, which we will call the affine Hessian $AH(\phi)$,  can be written (cf. Example 11.4 of \cite{Gibson}) as
 
 \begin{equation*}
 AH(\phi) = {\begin{vmatrix}\frac{d}{d-1}\phi&\frac{\partial \phi}{\partial x}&\frac{\partial \phi}{\partial y}\\\frac{\partial \phi}{\partial x}&\frac{\partial^2 \phi}{\partial x^2}&\frac{\partial^2 \phi}{\partial x \partial y}\\\frac{\partial \phi}{\partial y}& \frac{\partial^2 \phi}{\partial x \partial y}&\frac{\partial^2 \phi}{\partial y^2}\end{vmatrix}}. 
 \end{equation*} 
 To find the flexes of $\phi(x,y)=0$ one finds the common zeros of $\phi$ and $AH(\phi)$.  Hence, in finding the flexes we can work ``modulo $\phi$'' and so we can replace $AH(\phi)$ by the modified affine Hessian, which we will denote $\overline{AH}(\phi)$, obtained by replacing the (1,1)-entry in the expression for $AH(\phi)$ by 0.  Thus the flexes of $C_1$ in the $xy$-plane are the common zeros of $f_1$ and $\overline{AH}(f_1)$.
 
 We used the NSolve command of Mathematica to approximate the common zeros of $f_1$ and $\overline{AH}(f_1)$.  The result is that each of the three nodes is a Weierstrass point of weight 6, $(0:-1:1)$ is a Weierstrass point of weight 2 (as noted above), $(0:0:1)$ is a Weierstrass point of weight 1, and approximate values, to three decimal places, for other Weierstrass points of weight 1 are: $(0.325:-1.140:1), (-0.177 - 0164i:0.032 - 0.125i:1),$ and $(-0.177 + 0.164i: 0.032+0.125i:1)$.  This accounts for all 24 Weierstrass points (with multiplicities) of $C_1$.  We note that $C_1$ is a curve of Case 2 from \cite{LaxWid}.  
 
\section{The fiber over $t=\infty$}

 By replacing $t$ by $-1/s$, the fiber over $t=\infty$ in our original family is the same as the fiber over $s=0$ of the family $\widetilde{F}(X,Y,Z,s)=G(X,Y,Z) - sX^4=0$. (We have introduced a $-1$ coefficient for $s$ to agree with the notation in \cite{EstCD}.) From (\ref{Gfactor}), we see that this fiber is a reducible curve with a rational component and a nodal cubic component with a node at $(1:1:1)$. Of course, the Hessian of this curve vanishes on the rational component, but one can still determine the limit Weierstrass points on this fiber.
 
  We wish to use Example 5.3 of \cite{EstCD}, but in order to do that we need to perform a change of variables.  Define
  $$\widehat{F}(X,Y,Z,s)=\widetilde{F}(X-Y-Z,Y,Z,s).$$
  Then
  \begin{equation}\label{newfam}
    \widehat{F}(X,Y,Z,s)=X\widehat{G}(X,Y,Z) -  s(X-Y-Z)^4,
 \end{equation}
     where

$$\begin{array}{rcl}
\widehat{G}(X,Y,Z)&=&X^3 -7X^2Y +2X^2Z+12XY^2 -3XYZ-\\
&&6XZ^2-3Y^3 + 9YZ^2 +3Z^3.\\
\end{array}$$
  The families $\widetilde{F}=0$ and $\widehat{F}=0$ are isomorphic and we have
\begin{equation}\label{varchange}  
\widetilde{F}(a,b,c,s)=0 \Leftrightarrow \widehat{F}(a+b+c,b,c,s)=0. 
\end{equation}
  The curve $\widehat{G}=0$ is an irreducible nodal cubic with a node at $(3:1:1)$.

In Example 5.3 of \cite{EstCD}, Esteves gives a formula involving a sum of three 0-cycles for the 0-cycle $R$ of the limit on the fiber over $s=0$ of flexes on nearby smooth fibers for a family like (\ref{newfam}).  One of these summands is $3[\widehat{G}\cdot X]$; i.e., three times the intersection cycle of $\widehat{G}=0$ and $X=0$ (assuming these intersection points are smooth points of $\widehat{G}$).  Here, we have 

$$[\widehat{G}\cdot X]= (0:\gamma_1:1) + (0:\gamma_2:1) + (0:\gamma_3:1),$$
where the $\gamma_j$ are the three roots of $Y^3-3Y-1=0$.  Notice that these roots are the negatives of the roots of $x^3-3x+1=0$ that appeared above. To three decimal places, approximate values for these roots are $-1.532, -0.347$, and $1.879$.  

A second summand, which, after Esteves, we denote $R_{\widehat{G}}$, is the 0-cycle of flexes of $\widehat{G}=0$.  Let
$$ \hat{f}=x^3 -7x^2y+2x^2 +12xy^2-3xy-6x-3y^3+9y+3$$
denote the dehomogenization of $\widehat{G}$.  Finding the common zeros of $\hat{f}$ and $\overline{AH}(\hat{f})$, we get that
$$R_{\widehat{G}} = 6(3:1:1) +  (1:0:1) + (\alpha:\beta:1) + (\bar{\alpha}:\bar{\beta}:1),$$
where, to three decimal places, $\alpha\approx 1.643 + .124i$ and $\beta\approx 0.429 + 0.083i$.

The third summand is a bit more complicated.  Put $\hat{g}(Y,Z)= \widehat{G}(0,Y,Z)$ and put $\bar{f}(Y,Z)=(-Y-Z)^4$ (this comes from setting $X=0$ in the function multiplying $s$ in the family $\widehat{F}=0$).  One now considers the following determinant, reminiscent of a Hessian:

 \begin{equation*}
 \hat{h}(Y,Z) = {\begin{vmatrix}\frac{\partial^2 \bar{f}}{\partial Z^2}&\frac{\partial^2 (Y\hat{g})}{\partial Z^2}&\frac{\partial^2 (Z\hat{g})}{\partial Z^2}\\ \\ \frac{\partial^2 \bar{f}}{\partial Y \partial Z}&\frac{\partial^2 (Y\hat{g})}{\partial Y \partial Z}&\frac{\partial^2 (Z\hat{g})}{\partial Y \partial Z}\\ \\ \frac{\partial^2 \bar{f}}{\partial Y^2}& \frac{\partial^2 (Y\hat{g})}{\partial Y^2}&\frac{\partial^2 (Z\hat{g})}{\partial Y^2}\end{vmatrix}}. 
 \end{equation*} 
 The third summand, which Esteves denotes $R_X$, in the expression for $R$ is then the 0-cycle of ${\mathbb{P}}^2$ cut out on $X=0$ by $\hat h(Y,Z)=0$.
 
 In our case, we find that
  \begin{align*}
 \hat{h}(Y,Z) &= 972(Y+Z)^2{\begin{vmatrix}1&2Y(Y+Z)&2Z(3Y+2Z)\\1&Z(4Y+Z)&-Y^2 + 3Z^2\\1&2(-2Y^2+Z^2)&-2YZ\end{vmatrix}}\\
 &= -1944(Y+Z)^2 (3Y^4 +11Y^3Z +12Y^2Z^2 +3YZ^3 +Z^4). 
 \end{align*} 
 We then have that 
 $$R_X = 2(0:-1:1)+ (0:\xi:1) + (0:\bar{\xi}:1) + (0:\mu:1) +(0:\bar{\mu}:1),$$
 where $\xi,\bar{\xi},\mu,\bar{\mu}$ are the roots of $3Y^4+11Y^3+12Y^2 +3Y +1=0$. To three decimal places, we have $\xi\approx -1.734 + 0.446i$ and $\mu\approx -0.100 + 0.307i$. 
 
Now, using (\ref{varchange}), we get the following result.
\begin{theorem}\label{Zcycle}
The 0-cycle of limit Weierstrass points on the fiber over $s=0$ of the family $G(X,Y,Z) - sX^4=0$ is

\begin{align*}
& 6(1:1:1) + 3\sum_{j=1}^3 (-\gamma_j-1:\gamma_j:1) + 2(0:-1:1)+\\
&(0:0:1) +(\alpha-\beta-1:\beta:1) + (\bar{\alpha}-\bar{\beta}-1:\bar{\beta}:1)+\\
&(-\xi -1:\xi:1) + (-\bar{\xi}-1:\bar{\xi}:1) + (-\mu -1:\mu:1) +(-\bar{\mu}-1:\bar{\mu}:1).
\end{align*}
This is also the 0-cycle of limit Weierstrass points  on the fiber over $t=\infty$ of the family $F=0$.
\end{theorem} 
 
We can check this result by computing the Weierstrass points (flexes) on a fiber of $F=0$ over a large value of $t$. Put
$$f_m(x,y) = F(X,Y,1,10^6).$$
We computed estimates for the flexes on $f_m(x,y)=0$ by using the NSolve command in Mathematica to find approximate values for the common zeros of $f_m$ and $\overline{AH}(f_m)$.  As we already know, $(0,-1)$ is a hyperflex and $(0,0)$ is a flex. There are 21 other flexes for which we give approximate values to three decimal places and we give the coordinates in the $xy$-plane for the point in the 0-cycle in Theorem \ref{Zcycle} that these flexes are approaching:

\begin{align*}
&(1.025,1.020),(1.010,1.005)\approx (1,1)\\&(0.995 -0.008i,0.998 -0.004i),(0.995+0008i,0.998+0.004i)\approx (1,1)\\&(0.988-0.020i,0.990-0.016i),(0.988+0.020i,0.990+0.016)\approx (1,1)\\
&(0.529,-1.527),(0.533 +0.002i,-1.534-0.004i),\\
&(0.533-0.002i,-1.534+0.04i)\approx (-\gamma_1-1,\gamma_1)\\
&(-0.665,-0.357),(-0.647-0.012i,-0.342-0.008i),\\&(-0.647+0.010i,-0.342+.008i)\approx (-\gamma_2-1,\gamma_2)\\
&(-2.874,1.857),(-2.882+0.005i,1.891-0.019i),\\
&(-2.882-0.005i,1.891+0.019i)\approx (-\gamma_3-1,\gamma_3)\\
&(0.214 +0.041i,0.429+0.082i)\approx (\alpha-\beta-1,\beta)\\
&(0.214 -0.041i,0.429-0.082i)\approx (\bar{\alpha} -\bar{\beta} -1,\bar{\beta})\\
&(0.733 - 0.446i,-1.733 + 0.446i)\approx (-\xi -1,\xi)\\
&(0.733 + 0.446i, -1.733 -0.446i)\approx (-\bar{\xi}-1,\bar{\xi})\\
&(-0.900 -0.307i,-0.100 + 0.307i)\approx (-\mu -1,\mu)\\
&(-0.900 + 0.307i, -0.100 -0.307i)\approx (-\bar{\mu}-1,\bar{\mu}) 
\end{align*}

\section{The fiber over $t=0$}

To find the limit Weierstrass points on the fiber over $t=0$ (i.e., $X^4=0$) of our family (\ref{family}) we will use ideas from section IV.1.3 of \cite{GeomSch}.  Put $B= {\mathbb A}^1_{\mathbb{C}}= \hbox{ Spec }{\mathbb{C}}[t]$ and let ${\mathcal{C}}$ denote the curve over $B$ given by our family (\ref{family}).  Put $B^*=\hbox{ Spec }{\mathbb{C}}[t,t^{-1}]$, the punctured affine line.

Let $\mathcal{W}$ denote the scheme of flexes of $\mathcal{C}$. As in \cite{GeomSch}, let $\mathcal{W}^*$ be the inverse image in $\mathcal{W}$ of $B^*$.  The scheme $\mathcal{W}^*$ is finite and flat over $B^*$, with the fiber of $\mathcal{W}^*$ over a (nonzero) value of $t$ consisting of the 24 flexes (with multiplicity) on the fiber of $\mathcal{C}$ over that value of $t$.  (This is true even for $t=1$.)

Let $\mathcal{W}'$ denote the closure of $\mathcal{W}^*$ in ${\mathbb{P}}^2_B$, and let $W'_0$ denote the fiber of $\mathcal{W}'$ over $t=0$.  Since $B$ is nonsingular and one-dimensional, it follows from III.9.8 of \cite{Hart} that $\mathcal{W}'$ is flat over all of $B$.  Hence $W'_0$ has dimension 0 and degree 24 over $\mathbb{C}$.  Since we have shown that $(0:-1:1)$ is a hyperflex and $(0:0:1)$ is a flex on every fiber of $\mathcal{W}^*$ over $B^*$, we may infer that $(0:-1:1)$ is a point of multiplicity at least 2 in $W_0'$ and $(0:0:1)$ is a point in $W_0'$.

Let $I$ denote the ideal of $\mathcal{W}$ in the affine open subset Spec ${\mathbb{C}}[t] [x,y]\cong {\mathbb{A}}^2_B\subset {\mathbb{P}}^2_B$.  Then $$I = (f(x,y,t),\overline{AH}(f(x,y,t))),$$ where $f(x,y,t)=F(x,y,1,t)$ is the dehomogenization of $F$.
  Then, as in \cite{GeomSch}, we have that, in this affine open subset, 
  \begin{align*}
  &I^* = I \cdot \mathbb{C}[t,t^{-1}][x,y]\hbox{ is the ideal of }\mathcal{W}^*\\
  &I'= I^* \cap \mathbb{C}[t][x,y]\hbox{ is the ideal of }\mathcal{W}'\\
  &I'_0 = (I',t)\hbox{ is the ideal of }W'_0\hbox{ over }t=0.
  \end{align*}
  
  We will compute these ideals, over the rationals, using Macaulay2.  First, we construct the ideal $I$ with the following inputs: 
 \\
 \\
 {\tt 
 R=QQ[x,y,t]\\
 G=(x+y+1)*(x\verb|^|3-4*x\verb|^|2*y+5*x\verb|^|2+x*y\verb|^|2-7*x*y+x+3*y\verb|^|3)\\
 F=x\verb|^|4+t*G\\
 F1=diff(x,F)\\
 F2=diff(y,F)\\
 F12=diff(x,F2)\\
 F11=diff(x,F1)\\
 F22=diff(y,F2)\\
 M=matrix\verb|{|\verb|{|0,F1,F2\verb|}|,\verb|{|F1,F11,F12\verb|}|,\verb|{|F2,F12,F22\verb|}|\verb|}|\\
 AHbar=det M\\
 I=ideal(F,AHbar)\\
 }
 
Now we want to compute $I^*$, the extension of $I$ to the ring $\mathbb{Q}[t,t^{-1}][x,y]$.  We first create this extension ring as a quotient ring $\mathbb{Q}[x,y,t,s]/(st-1)$ so that $s$ plays the role of $t^{-1}$. Then we consider the embedding $\psi$ of our original ring into this new ring, then take the image of $I$ under this mapping. Next, we compute $I'$ by taking the preimage of $I^*$ under this embedding. The Macaulay2 inputs are:
\\
\\
{\tt
S=R[s]/(s*t-1)\\
psi=map(S,R,\verb|{|x,y,t\verb|}|)\\
Istar=psi I\\
Iprime=preimage\verb|_|psi Istar
}
\\

Then we compute the ideal $I_0'$:
\\
\\
{\tt use(R)\\
Izeroprime=Iprime+ideal(t)\\
gens gb Izeroprime
}
\\

A Groebner basis for $I_0'$ (with respect to graded reverse lex) is $\{t,x^4,g_1,g_2,g_3,g_4\}$ where

\begin{align*}
g_1 = &\> 8x^3y^3 + 3x^2y^4 - 8x^3y^2 -4x^3y\\
g_2 = &\> 3x^2y^5 +3x^2y^4 -12x^3y^2 -4x^3y\\
g_3 = &\>  48xy^7 +12y^8 -24xy^6 +12y^7 +45x^2y^4 -72xy^5 -448x^3y^2+\\
      &\> 144x^2y^3-12xy^4-196x^3y +24x^2y^2+12x^3-6x^2y\\
g_4 = &\> 24y^9 +60y^8 -216xy^6 +36y^7 +525x^2y^4 -240xy^5 -1056x^3y^2+\\
      &\> 480x^2y^3 -36xy^4 -524x^3y +60x^2y^2+36x^3 -18x^2y.            
\end{align*}   
 
 From this Groebner basis, we see that, in the affine open set in which we are working, the scheme $W_0'$ is supported only at the two points $(0,0)$ and $(0,-1)$.  By \cite{SS}, we can compute the multiplicity of $W_0'$ at $(0,0)$ by using the ``saturate'' command in Macaulay2 as follows:
 \\
 \\
 {\tt degree(Izeroprime : saturate(Izeroprime))}
 \\
 \\
 and we get 22 as the multiplicity at the origin.  Since we know that the multiplicity of $W_0'$ is at least 2 at $(0,-1)$ and the total degree of $W_0'$ is 24, we conclude that $W_0'$ is not supported at any points outside our affine subset and we have shown the following result.
 
 \begin{theorem}
 The 0-cycle of limit Weierstrass points on the fiber $X^4=0$ over $t=0$ of our family $F=0$ is $22(0:0:1) + 2(0:-1:1)$.
 \end{theorem}  
 
 We note that when we tried to check this result using NSolve to estimate the values of the flexes on the fiber over $t=10^{-7}$ by finding common roots of $f_\epsilon(x,y)=F(x,y,1,10^{-7})$ and $\overline{AH}(f_\epsilon)$, there were errors (for example, NSolve found more than 24 approximate roots, some of them very large).  However, when we used NSolve to estimate flexes on this same curve by estimating common zeros of $F_\epsilon(X,Y,Z)=F(X,Y,Z,10^{-7})$ and the (projective) Hessian of $F_\epsilon$ we did get appropriate answers: $(0:-1:1)$ a hyperflex, $(0:0:1)$ a flex, and 21 other flexes in the affine $xy$-plane whose $x$-coordinates were $0.00$ to two decimal places and whose $y$-coordinates had modulus less than $0.3$ and usually less than $0.1$. The convergence to the 0-cycle in the theorem appears to be ``slow.''  
   
\bibliographystyle{amsplain}

\bibliography{mybibliography}

\providecommand{\bysame}{\leavevmode\hbox to3em{\hrulefill}\thinspace}
\providecommand{\MR}{\relax\ifhmode\unskip\space\fi MR }
\providecommand{\MRhref}[2]{%
  \href{http://www.ams.org/mathscinet-getitem?mr=#1}{#2}
}
\providecommand{\href}[2]{#2}
\begin{thebibliography}{10}

\bibitem{coskap}
Matteo Costantini and Andr\'{e} Kappes, \emph{The equation of the
  {K}enyon-{S}millie {$(2,3,4)$}-{T}eichm\"{u}ller curve}, J. Mod. Dyn.
  \textbf{11} (2017), 17--41. \MR{3588522}

\bibitem{EHLimitsWpts}
David Eisenbud and Joe Harris, \emph{Existence, decomposition, and limits of
  certain {W}eierstrass points}, Invent. Math. \textbf{87} (1987), no.~3,
  495--515. \MR{874034}

\bibitem{GeomSch}
\bysame, \emph{The geometry of schemes}, Graduate Texts in Mathematics, vol.
  197, Springer-Verlag, New York, 2000. \MR{1730819}

\bibitem{EstCD}
Eduardo Esteves, \emph{Limits of {C}artier divisors}, J. Pure Appl. Algebra
  \textbf{214} (2010), no.~10, 1718--1728. \MR{2608101}

\bibitem{EstMed}
Eduardo Esteves, Nivaldo Medeiros, and Wallace Sousa, \emph{Limits of dual
  curves via foliations}, arXiv:1909.08647, 2019.

\bibitem{Gibson}
C.~G. Gibson, \emph{Elementary geometry of algebraic curves: an undergraduate
  introduction}, Cambridge University Press, Cambridge, 1998. \MR{1663524}

\bibitem{M2}
Daniel~R. Grayson and Michael~E. Stillman, \emph{Macaulay2, a software system
  for research in algebraic geometry}, Available at {\tt
  http://www.math.uiuc.edu/Macaulay2/}.

\bibitem{Hart}
Robin Hartshorne, \emph{Algebraic geometry}, Springer-Verlag, New
  York-Heidelberg, 1977, Graduate Texts in Mathematics, No. 52. \MR{0463157}

\bibitem{LaxPAMS}
R.~F. Lax, \emph{Weierstrass weight and degenerations}, Proc. Amer. Math. Soc.
  \textbf{101} (1987), no.~1, 8--10. \MR{897062}

\bibitem{LaxWid}
R.~F. Lax and Carl Widland, \emph{Weierstrass points on rational nodal curves
  of genus {$3$}}, Canad. Math. Bull. \textbf{30} (1987), no.~3, 286--294.
  \MR{906350}

\bibitem{Moller}
Martin M\"{o}ller, \emph{Geometry of {T}eichm\"{u}ller curves}, Proceedings of
  the {I}nternational {C}ongress of {M}athematicians---{R}io de {J}aneiro 2018.
  {V}ol. {III}. {I}nvited lectures, World Sci. Publ., Hackensack, NJ, 2018,
  pp.~2017--2034. \MR{3966840}

\bibitem{SS}
Gregory~G. Smith and Bernd Sturmfels, \emph{Teaching the geometry of schemes},
  Computations in algebraic geometry with {M}acaulay 2, Algorithms Comput.
  Math., vol.~8, Springer, Berlin, 2002, pp.~55--70. \MR{1949548}

\bibitem{WidLax}
Carl Widland and Robert Lax, \emph{Weierstrass points on {G}orenstein curves},
  Pacific J. Math. \textbf{142} (1990), no.~1, 197--208. \MR{1038736}

\end{thebibliography}

\end{document}